\newcommand{\pdiv}{$p$-divisible }

\newcommand{\G}{\mathbb{G}}
\newcommand{\F}{\mathbb{F}}

\newcommand{\Q}{\mathbb{Q}}

\newcommand{\Z}{\mathbb{Z}}
\newcommand{\QZ}{\Q_p/\Z_p}

\newcommand{\al}{\alpha}
\newcommand{\Lk}{\Lambda_k}

\newcommand{\lra}[1]{\overset{#1}{\longrightarrow}}

\newcommand{\Prod}[1]{\underset{#1}{\prod}}

\newcommand{\Coprod}[1]{\underset{#1}{\coprod}}
\newcommand{\Colim}[1]{\underset{#1}{\colim}}

\newcommand{\E}{E_{n}}
\newcommand{\LE}{L_{t,n}}

\newcommand{\Loops}{\mathcal{L}}

\newcommand{\bC}{\bar{C}_{t}}

\date{\today}
\documentclass[10pt]{amsart}
\usepackage{amssymb,amsfonts,amsthm,amsmath,verbatim,lscape,url, color}
\usepackage{enumerate}
\usepackage{mathrsfs}
\usepackage{stmaryrd}
\usepackage{graphicx}
\usepackage[margin=1.5in]{geometry}
\usepackage[enableskew]{youngtab}
\usepackage[all]{xy}
\theoremstyle{definition}

\DeclareMathOperator{\im}{im}
\DeclareMathOperator{\Fix}{Fix}
\DeclareMathOperator{\Hom}{Hom}
\DeclareMathOperator{\colim}{colim}

\DeclareMathOperator{\Iso}{Iso}
\DeclareMathOperator{\Spec}{Spec}

\newcommand{\Mod}{\mathrm{Mod}}

\newcommand{\Sp}{\mathrm{Sp}}

\newcommand{\Es}[1]{E_{#1}}
\newcommand{\As}[1]{A_{#1}}
\newcommand{\fm}{\mathfrak{m}}
\newcommand{\fib}{\mathrm{fib}}
\newcommand{\cof}{\mathrm{cofib}}

\newcommand{\Kos}{\mathrm{Kos}}

\newcommand{\cC}{\mathcal{C}}
\newcommand{\sm}{\wedge}
\newcommand{\W}{\mathrm{W}}
\newcommand{\Good}{\mathcal{G}}

\newtheorem{theorem}[subsection]{Theorem}
\newtheorem{proposition}[subsection]{Proposition}
\newtheorem{corollary}[subsection]{Corollary}
\newtheorem{lemma}[subsection]{Lemma}

\newtheorem*{maintheorem}{Theorem}

\newtheorem{definition}[subsection]{Definition}

\newtheorem{example}[subsection]{Example}

\newtheorem{remark}[subsection]{Remark}

\makeatletter
\ifx\SK@label\undefined\let\SK@label\label\fi
 \let\your@thm\@thm
 \def\@thm#1#2#3{\gdef\currthmtype{#3}\your@thm{#1}{#2}{#3}}
 \def\mylabel#1{{\let\your@currentlabel\@currentlabel\def\@currentlabel
  {\currthmtype~\your@currentlabel}
 \SK@label{#1@}}\label{#1}}
 \def\myref#1{\ref{#1@}}
\makeatother

\begin{document}
\title{Centralizers in good groups are good}
\author{Tobias Barthel}
\author{Nathaniel Stapleton}

\begin{abstract}
We modify the transchromatic character maps of \cite{tgcm} to land in a faithfully flat extension of Morava $E$-theory. Our construction makes use of the interaction between topological and algebraic localization and completion. As an application we prove that centralizers of tuples of commuting prime-power order elements in good groups are good and we compute a new example. 
\end{abstract}

\maketitle

\section{Introduction and outline}
The character maps of \cite{tgcm} suggest the intriguing possibility of approximating height $n$ Morava $E$-theory by Morava $E$-theory of lower height. In particular, it is easy to imagine that a character map from $\E$ to $p$-adic $K$-theory could have many applications because it could reduce height $n$ problems to representation theory in the same way that the character map of \cite{hkr} reduces height $n$ problems to combinatorial problems. However, the maps produced in \cite{tgcm} have codomain an extension of the $K(t)$-localization of $\E$ for $t < n$. This cohomology theory is less familiar and presents some computational difficulties because its coefficients are not a complete local ring. In this paper, we present a modification of the character maps that, for good groups, land in a faithfully flat extension of $E_t$. This work grew out of work of the second author with Tomer Schlank on Strickland's theorem in \cite{tpst} where this modified character map from height $n$ to height $1$ plays a critical role. 

A finite group is good (at a prime $p$) if $\E^*(BG)$ is free and evenly concentrated. There are a variety of classes of groups that are known to be good. These include finite abelian groups, symmetric groups, finite general linear groups away from the characteristic, wreath products and products of these groups, and groups of order $p^3$ (see Theorem E and Proposition 7.10 in \cite{hkr}, \cite{tanabe}). It was a conjecture of Hopkins, Kuhn, and Ravenel that all groups are good; however this was disproved by Kriz in \cite{Krizodd}. A corollary of the construction of this modified character map is that centralizers of 
abelian $p$-groups in good groups are good. This enlarges the class of groups known to be good in a very different way than the results above. 

Our approach relies on a variety of facts concerning completion and localization of flat modules in stable homotopy theory. In order to put these into a more abstract context, we review the relationship between chromatic localization functors on the category of $MU$-modules and certain arithmetic localizations and completions (as described in \cite{greenleesmay95}). As an immediate consequence we obtain that the coefficients of the $K(t)$-localization of a flat $\E$-module $M$ are given by the simple formula
\[
\pi_*(L_{K(t)}M) \cong (\pi_*M)[u_{t}^{-1}]_{(p, \ldots, u_{t-1})}^{\wedge}.
\]
We also provide a proof of a mild generalization of Hovey's unpublished theorem that the $I_n$-completion of a flat $\E^*$-module is flat. 

We then use these methods to analyze the spectrum $\bC = L_{K(t)}(C_t \sm E_t)$, where $C_t$ is an $E_{\infty}$-ring with coefficients the ring $C_{t}^*$ from \cite{tgcm}. After proving that $\bC^*$ is faithfully flat as an $E_{t}^*$-module we construct the modified character maps. For $H$ the centralizer of a tuple of commuting elements in $G$, we show that there is an isomorphism
\[
\bC^* \otimes_{L_{K(t)}\E^*}L_{K(t)}\E^*(BH) \cong \bC^* \otimes_{E_{t}^*} E_{t}^*(BH).
\]

Now let $\Loops(-) = \hom(B\Z_p,-)$ be the $p$-adic free loop space functor. The main object of study is the composite of the character map from \cite{tgcm} with the isomorphism above
\[
\E^*(BG) \lra{} \bC^* \otimes_{L_{K(t)}\E^*}L_{K(t)}\E^*(\Loops^{n-t} BG) \cong \bC^* \otimes_{E_{t}^*} E_{t}^*(\Loops^{n-t} BG).
\]
The codomain of this character map is just the $E_t$-cohomology of $\Loops^{n-t} BG$ base changed to a faithfully flat extension. Morava $E_t$ is certainly more computable and more familiar than $L_{K(t)}\E$. Our main result gives a condition for when this map induces an isomorphism after base change to $\bC^*$:
\begin{maintheorem}
For a good group $G$ the map above, base changed to $\bC^*$, gives an isomorphism
\[
\bC^* \otimes_{E_{n}^*} E_{n}^*(BG) \lra{\cong} \bC^* \otimes_{E_{t}^*} E_{t}^*(\Loops^{n-t}BG).
\] 
\end{maintheorem}

This allows us to reduce certain height $n$ problems to height $t$ problems without introducing more exotic cohomology theories. For instance, an argument using faithfully flat descent for finitely generated projective modules proves \myref{goodcentralizers} (centralizers in good groups are good) from this result. The paper ends with a brief summary of what is known about good groups and a new example. 

\subsection*{Acknowledgements} We thank Omar Antol\'in Camarena, Mark Behrens, Tyler Lawson, Eric Peterson, Tomer Schlank, and Bj{\"o}rn Schuster for many helpful conversations. The second author was partially supported by NSF grant DMS-0943787.

\section{Arithmetic localization and completion}\label{arithmeticlocalization}

In this section we summarize the Greenlees--May theory of localization and completion in topology, as developed in \cite{greenleesmay95}, using some insights from \cite{dag12}.

\subsection{Construction of localization and completion}

Let $R$ be an $\Es{\infty}$-ring spectrum, so that the category $\Mod_R$ of $R$-modules has a symmetric monoidal structure, where the smash product over $R$ will be denoted $\sm$. Note that most of the arguments work as well for $\Es{2}$-ring spectra, but we will not need this extra generality here. 

Let $I \subseteq \pi_*R$ be a finitely generated ideal with a minimal set of generators $\{x_1,\ldots,x_n \}$. It is possible to weaken this hypothesis as in \cite{greenleesmay92}, but for simplicity we will restrict ourselves to the finitely generated situation. 

\begin{definition}
A module $M \in \Mod_R$ is called $I$-nilpotent if $I \subseteq \mathrm{supp}(m) = \{r \in \pi_0R\mid \exists n\colon r^nm=0\}$ for all $m \in \pi_*M$. This condition is equivalent to $M[1/x] = 0$ for all $x \in I$.
\end{definition}

If we define the Koszul complex as $\Kos(I) = \bigwedge_{i=1}^n\Kos(x_i)$ with $\Kos(x_i) = \fib(R \to R[1/x_i])$, then we can construct the fundamental cofiber sequence

\[\Kos(I) \longrightarrow R \longrightarrow \check C(I). \]

\begin{definition}
The right adjoint to the inclusion functor $\Mod_R^{I-\text{nil}} \hookrightarrow \Mod_R$ is given by $\Gamma_I = - \otimes \Kos(I)$. 
\end{definition} 

It follows easily that $\Mod_R^{I-\text{nil}}$ is compactly generated by $\bigwedge_{i=1}^n\, \cof(R \xrightarrow{x_i} R)$. For example, $\Mod_{S_{(p)}^0}^{(p)-\text{nilp}} \simeq \Mod_{S_{(p)}^0}^{\Q-\text{acyclic}}$ is generated as a localizing subcategory by $S^0/p$.

If $\cC$ is a full stable subcategory of $\Mod_R$, its left orthogonal is defined as the full subcategory of $\Mod_R$ on those objects $N$ for which $\Hom(M,N) = 0$ for all $M \in \cC$. 

\begin{definition}
$\Mod_R^{I-\text{loc}}$ is the left orthogonal to $\Mod_R^{I-\text{nil}}$, i.e., $M \in \Mod_R^{I-\text{loc}}$ if and only if $\Hom_R(N,M)=0$ for all $N \in \Mod_R^{I-\text{nil}}$.
\end{definition}

\begin{lemma}
The inclusion functor $\Mod_R^{I-\text{loc}} \hookrightarrow \Mod_R$ admits a left adjoint, given by $I$-localization $L_I= - \sm \check C(I)$, also written as $(-)[I^{-1}]$. In particular, this gives rise to a fiber sequence of functors
\[\Gamma_I \longrightarrow \mathrm{id} \longrightarrow L_I.\]
\end{lemma}

\begin{definition}
$\Mod_R^{I-\text{comp}}$ is left orthogonal to $\Mod_R^{I-\text{loc}}$. 
\end{definition}

Equivalently, an $R$-module $M$ is $I$-complete if and only if $\lim (\ldots \xrightarrow{x} M \xrightarrow{x} M) =0$ for all $x \in I$, as shown in \cite[Cor. 4.2.8]{dag12}

\begin{lemma}
The $I$-adic completion functor is defined as $(-)_I^{\wedge} = \Hom_R(\Kos(I),-)$, and it is left adjoint to the inclusion functor $\Mod_R^{I-\text{comp}} \hookrightarrow \Mod_R$.
\end{lemma}

As a special case, $(-)_{(x)}^{\wedge} \simeq \lim_s (-/x^s)$, which coincides with the familiar construction of completion. 

The following diagram summarizes our discussion of arithmetic localization and completion,

\[\xymatrix{& \Mod_R^{I-\text{loc}} \ar@<0.5ex>[d] \ar@{-->}@/^1.5pc/[ddr] \\
& \Mod_R \ar@<0.5ex>[u]^{L_I} \ar@<0.5ex>[ld]^{\Gamma_I} \ar@<0.5ex>[rd]^{(-)_I^{\wedge}} \\
\Mod_R^{I-\text{nil}} \ar@<0.5ex>[ru] \ar[rr]_{\sim} \ar@{-->}@/^1.5pc/[ruu] & & \Mod_R^{I-\text{comp}} \ar@<0.5ex>[lu]}\]
where the dotted arrows indicate left orthogonality. This reduces to the usual derived functors of localization and completion upon specialization to Eilenberg--Mac Lane spectra.

\subsection{Bousfield localization}

Following \cite{ekmm}, we will work in the category of $\Mod_R$ of modules over an $\Es{\infty}$-ring spectrum $R$. Let $E$ be an $R$-module. 

\begin{definition}
An $R$-module $X$ is called $E$-acyclic if $X \sm E = 0$, and $Y \in \Mod_R$ is called $E$-local if $\Hom(X,Y) = 0$ for all $E$-acyclic $X$. A morphism $f$ is an $E$-equivalence if $f \sm E$ is an equivalence. 
\end{definition}

The following fundamental result was proven by Bousfield \cite{Bousfieldlocalization}.

\begin{theorem}
There exists a functor $L_E\colon \Mod_R \to \Mod_R$ together with a natural transformation $\mathrm{id} \to L_E$ such that $X \to L_EX$ is an $E$-equivalence with $E$-local target for all $X$. Equivalently, $X \to L_EX$ is the initial map into an $E$-local object. 
\end{theorem}

Recall also that a localization functor $L$ is called smashing if, for all $M \in \Mod_R$, $LM = M \sm LR$. As in \cite{greenleesmay95}, we now identify the arithmetic localization and completion functors encountered earlier as special cases of Bousfield localization. 

\begin{proposition}\mylabel{locbl}
Let $R$ be an $\Es{\infty}$-ring spectrum with Noetherian coefficients and let $I$ be an ideal in $\pi_*R$, then the following holds.  
\begin{enumerate}
 \item $L_I$ is the smashing Bousfield localization with respect to $\check{C}(I)$.
 \item There is a spectral sequence $E_{p,q}^2 = \check CH_I^{-p,-q}(\pi_*M) \Rightarrow \pi_{p+q}(L_IM)$. Here, $\check CH_I^*$ denotes $\mathrm{\check{C}ech}$ cohomology with respect to $I$ as defined in \cite{greenleesmay95}.
\end{enumerate}
\end{proposition}

We will be mainly interested in completion. Recall that algebraic $I$-completion is not exact on the category of all $R_*$-modules, but we can consider its left derived functors $L_s^I = \mathbb{L}^s(-)_I^{\wedge}$. These will be studied in more detail in Section \ref{lcompletion}.

\begin{proposition}\mylabel{compbl}
Let $R$ be a commutative $S$-algebra with Noetherian coefficients and let $I$ be an ideal in $\pi_*R$, then the following holds.  
\begin{enumerate}
 \item $(-)_I^{\wedge}$ is Bousfield localization with respect to $\Kos(I)$. In general, $(-)_I^{\wedge}$ is not smashing. 
 \item There is a spectral sequence $E_{s,t}^2 =  L_s(\pi_*M)_t \Rightarrow \pi_{s+t}(M_I^{\wedge})$, where $L_s$ denotes the $s$-th left derived functor of ordinary $I$-adic completion.   
\end{enumerate}
\end{proposition}

\begin{remark}
More generally, the $E^2$-term of the above spectral sequence can be identified with the local homology of groups of $\pi_*M$ with respect to $I$, $E_{s,t}^2 = H_{s,t}^I(\pi_*M)$.
\end{remark}

\section{Localization and completion of $MU$-modules}

The goal of this section is to show that the restrictions of certain Bousfield localization functors appearing in chromatic homotopy theory to $MU$-modules can be expressed as combinations of the arithmetic functors of Section \ref{arithmeticlocalization}. This is certainly well-known to experts, but since there is no published reference for these results, we include the proofs. 

Moreover, the same techniques allow us to study the effect of $K(n)$-localization on coefficients, which admits an explicit description for flat modules. 

\subsection{Recollections}

Fix a prime $p$, and let $E_n$ and $K(n)$ denote Morava $E$-theory and Morava $K$-theory at height $n$, respectively. Recall that $E_n$ is a Landweber exact $\Es{\infty}$-ring spectrum with coefficients $E_n^*=\W k\llbracket u_1,\ldots,u_{n-1}\rrbracket [u^{\pm 1}]$, where $\W k$ is the ring of Witt vectors of $k = \F_{p^n}$ and $u$ has degree $2$. The spectrum representing Morava $K$-theory is a complex orientable $\As{\infty}$-ring spectrum with $K(n)_* = \F_{p^n}[v_n^{\pm 1}]$ with $v_n$ of degree $2(p^n-1)$.

These spectra come with associated Bousfield localization functors $L_n$ and $L_{K(n)}$ that play a fundamental role in the chromatic approach to stable homotopy theory. We recall two important relations between these functors. 

\begin{itemize}
 \item $L_n = L_{\bigvee_{i=0}^n K(i)}$
 \item There is a homotopy pullback square of functors on spectra
 \[\xymatrix{L_{n} \ar[r] \ar[d] & L_{K(n)} \ar[d] \\
 L_{n-1} \ar[r] & L_{n-1}L_{K(n)},}\]
 usually called the chromatic fracture square. 
\end{itemize}

The $n$-th monochromatic layer $M_n\colon \Sp \to \Sp$ is defined as the fiber $M_n = \fib(L_n \to L_{n-1})$. By the smash product theorem of Hopkins and Ravenel, $L_n$ and hence $M_n$ are smashing for all $n$, whereas $L_{K(n)}$ does not have this property. Moreover, Hovey and Strickland provide a convenient description of $K(n)$-localization.

\begin{proposition}\mylabel{kchar}
For any spectrum $X$ and $n\ge0$, there is an equivalence
\[L_{K(n)}X = \Hom_{S^0}(M_nS^0,L_nX).\]
\end{proposition}

\subsection{Identification of chromatic functors}

In \cite{greenleesmay95}, Greenlees and May provide the starting point of a dictionary between arithmetic and chromatic localization and completion functors on the category of $MU$-modules. Since $BP$ is known to be $\Es{4}$, we could work with $BP$ as well. 

\begin{proposition}\mylabel{chromtoarith}
For $N \in \Mod_{MU}$ and any $t \ge 0$:
\begin{enumerate}
 \item $L_tN = N[I_{t+1}^{-1}] \simeq N \sm \check{C}(I_{t+1})$.
 \item $N_{I_{t}}^{\wedge} \simeq \Hom_{S^0}(\colim_iM_i,N)$, where the $M_i$ form a cofinal sequence of generalized type $t$ Moore spectra. 
\end{enumerate}
Here, $I_t$ denotes the ideal $(p,v_1,\dots, v_{t-1})$.
\end{proposition}

\begin{remark}
The obvious analogue of this result hold for the category of $E_n$-modules as well.
\end{remark}

\begin{lemma}
For $N \in \Mod_{MU}$, $M_t(N) \simeq \Gamma_{I_t}(N[v_t^{-1}])$.
\end{lemma}
\begin{proof}
The following commutative diagram, in which all rows and columns are fiber sequences,
\[\xymatrix{\Kos(I_{t+1}) \ar[r] \ar[d] & MU \ar[r] \ar[d]^{\simeq} & \check C(I_{t+1}) \ar[d] \\
\Kos(I_t) \ar[r] \ar[d] & MU \ar[r] \ar[d] & \check C(I_t) \ar[d] \\
\Kos(I_t) \sm \check C(\nu_t) \ar[r] & 0 \ar[r] & \Sigma \Kos(I_t) \sm \check C(\nu_t)}\]
shows that $M_t(-)  = (-) \sm \Kos(I_t) \sm \check C(\nu_t)$,  since $L_tN \simeq N \sm \check C(I_{t+1})$.
\end{proof}

\begin{proposition}\mylabel{lknexplicit}
If $N \in \Mod_{MU}$, then $L_{K(t)}N \simeq (N[\nu_t^{-1}])_{I_t}^{\wedge}$. 
\end{proposition}
\begin{proof}
By \myref{kchar} and \myref{chromtoarith}, we have
\[L_{K(t)}N \simeq (N[I_{t+1}^{-1}])_{I_t}^{\wedge} \simeq \Hom(\Kos(I_t),N \sm \check C(I_{t+1})).\]
Consider the following commutative diagram of fiber sequences
\[\resizebox{\columnwidth}{!}{\xymatrix{\Hom(\Kos(I_t), N \sm \Kos(I_{t+1})) \ar[r] \ar[d]_{\phi} & \Hom(\Kos(I_t), N) \ar[r] \ar[d]^{\sim} & \Hom(\Kos(I_t), N \sm \check C(I_{t+1})) \ar[d] \\      
\Hom(\Kos(I_t), N \sm \Kos(\nu_t)) \ar[r] & \Hom(\Kos(I_t), N) \ar[r] & \Hom(\Kos(I_t), N \sm \check C(\nu_t)).}}\]
We claim that the first vertical map $\phi$ is an equivalence, hence so is the last one and the result follows. To see the claim, note that $\phi$ fits into a cofiber sequence
\[\xymatrix{\Hom(\Kos(I_t), N \sm \Kos(I_{t}) \sm \Kos(\nu_t)) \ar[r] & \Hom(\Kos(I_t), N \sm \Kos(\nu_t)) \ar[d] \\
& \Hom(\Kos(I_t), N \sm \check C(I_{t}) \sm \Kos(\nu_t))}\]
where the cofiber is contractible by adjunction, as there are no non-trivial maps from an $I_{t}$-local module to an $I_{t}$-complete module. 
\end{proof}

\begin{remark}
The spectral sequence of \myref{locbl} 
\[E_{s,t}^2 = \check CH_{I_{t+1}}^{-p,-q}(\pi_*N) \Rightarrow \pi_{p+q}(L_tN)\]
which corresponds to the geometric decomposition induced by the chromatic fracture cube. It accounts for the existence of odd-dimensional classes in the homotopy of $L_tE_n$, $0< t < n$.
\end{remark}

The following table summarizes the identifications of the chromatic functors on the category of $MU$-modules. 

\renewcommand{\arraystretch}{1.4}
\begin{center}
    \begin{tabular}{| l | l |}
    \hline 
    Chromatic functor & Arithmetic functor \\ \hline
    $L_t(-)$ & $(-)[I_{t+1}^{-1}]$ \\ 
    $C_t(-)$ & $\Gamma_{I_{t+1}}(-)$ \\ 
    $M_t(-)$ & $\Gamma_{I_t}(-)[\nu_t^{-1}]$ \\ 
    $L_{K(t)}(-)$ & $(-)[\nu_t^{-1}]_{I_t}^{\wedge}$ \\ \hline
    \end{tabular}
\end{center}

In particular, we have a commutative diagram
\[\xymatrix{& \Mod_{MU} \ar[d]^{[\nu_t^{-1}]} \ar@/^-1.5pc/[ldd]_{M_t} \ar@/^1.5pc/[rdd]^{L_{K(t)}} & \\
& \Mod_{\nu_t^{-1}MU} \ar@<0.5ex>[ld]^{\Gamma_{I_t}} \ar@<0.5ex>[rd]^{(-)_{I_t}^{\wedge}} & \\
\Mod_{\nu_t^{-1}MU}^{I_t-\text{nil}} \ar@<0.5ex>[ru] \ar[rr]_{\sim} & &  \Mod_{\nu_t^{-1}MU}^{I_t-\text{comp}} \ar@<0.5ex>[lu]}\]
where the bottom horizontal equivalence is the well-known equivalence between the height $t$ monochromatic category and the $K(t)$-local category when restricted to $MU$-modules.

\subsection{$K(n)$-localization and flatness}\label{lcompletion}

\myref{lknexplicit} can be used to compute the homotopy groups of localizations. Recall that a module spectrum $M$ over an $\Es{1}$-ring spectrum $R$ is said to be flat if and only if $\pi_*M$ is flat as a graded module over $\pi_*R$.

\begin{corollary}\mylabel{flatcoefficients}
If $N \in \Mod_{E_n}$ is flat, then $\pi_*L_{K(t)}N \cong ((\pi_*N)[\nu_t^{-1}])_{I_t}^{\wedge}$.
\end{corollary}
\begin{proof}
By \myref{lknexplicit}, $L_{K(t)}N \simeq (N[\nu_t^{-1}])_{I_t}^{\wedge}$. Since $N$ is flat, so is $N[\nu_t^{-1}]$, hence the spectral sequence of \myref{compbl} computing the completion collapses. Therefore, we get
\[\pi_*(N[\nu_t^{-1}])_{I_t}^{\wedge} \cong ((\pi_*N)[\nu_t^{-1}])_{I_t}^{\wedge},\]
since $\pi_*$ preserves filtered colimits. 
\end{proof}

More generally, \myref{compbl} gives a natural, strongly convergent spectral sequence 
\[E_{s,t}^2 = (L_s\pi_*M)_t \Rightarrow \pi_{s+t}L_{K(n)}M\]
with $E_{s,*}^2 = 0$ if $s > n$ and differentials $d^r: E_{s,t}^r \to E_{s-r,t+r-1}^r$. Using this spectral sequence, it is not hard to see \cite[Cor. 3.14]{frankland} that $M \in \Mod_{E_n}$ is $K(n)$-local if and only if $\pi_*M$ is isomorphic to $L_0(\pi_*M)$. In fact this holds more generally for any completion functor over a connective ring spectrum, see \cite[Thm. 4.2.13]{dag12}. 

\begin{remark}
This corollary complements Hovey's result for ring spectra, \cite[Thm. 1.5.4.]{Hovey-vn}.
\end{remark}

For the rest of this section, let $R$ be a regular complete local Noetherian commutative ring of dimension $n$, and let $I = (x_1,\ldots,x_t)$ be an ideal in $R$ with a chosen minimal regular sequence of generators. The main example of interest to us is $E_n^0 = \W k\llbracket u_1,\ldots u_{n-1} \rrbracket$ with its maximal ideal $\fm = (p,u_1,\ldots, u_{n-1})$.

By the Artin-Rees lemma, the algebraic completion functor $(-)_{I}^{\wedge}$ is exact when restricted to finitely generated modules, but it is neither left nor right exact in general. Therefore, for general $R$-modules, we have to consider the left derived functors $L_s$ of $I$-adic completion. However, $L_0$ coincides with ordinary $I$-adic completion for flat modules, so we may restrict ourselves to this case here. An overview of the construction and properties of these functors relevant to topology can be found in \cite{hoveystrickland, frankland, rezkanalytic}. 

\begin{remark}
The assumption that $R$ is Noetherian can be weakened. In particular, the theory applies as well to the non-Noetherian but coherent ring $BP_*=\Z_{(p)}[v_1,v_2,\ldots]$, see \cite{greenleesmay92}. 
\end{remark}

The following result was proven in the special case of $R={E_n}_*$ and $I = \fm$ by Hovey \cite{hoveyetheorynotes}; the arguments easily generalize to give the following flatness criterion. 

\begin{proposition}\mylabel{flatflat}
If $M$ is a flat $R$-module such that $M/I$ is projective over $R/I$, then $M_{I}^{\wedge}$ is also flat over $R$. 
\end{proposition}
\begin{proof}[Sketch of proof]
By \cite[Tag 05D3]{stacks-project}, the hypotheses imply that $M_{I}^{\wedge}$ is a retract of a pro-free module, i.e., a module of the form $F_{I}^{\wedge}$ with $F \in \Mod_R$ free. Since pro-free objects are retracts of products of $R$ by \cite[Prop. A.13]{hoveystrickland}\footnote{Note that the proof of Prop.~A.13 in \cite{hoveystrickland} generalizes to any finitely generated ideal $I$.}, hence flat as $R$ is Noetherian, it follows that $M_{I}^{\wedge}$ is also flat over $R$.
\end{proof}

\begin{remark}
In fact, in case $R$ is local and $I =\fm$ is the maximal ideal, the class of pro-free objects coincides with the collection of flat $R$-modules which are $I$-complete. This characterization does not generalize to arbitrary finitely generated ideals $I$, as the example $I=(0)$ shows. 
\end{remark}

\section{A short digression on Landweber exact theories}

We include a short digression on Landweber exact cohomology theories. This is partly to set up some technicalities that will be of use later and partly to clarify the relation between Landweber exactness and Brown representability. 

Assume that $E$ is a Landweber exact spectrum and $R$ is a flat $E$-module. Landweber exactness implies that
\[
R_*(X) \cong R_{*} \otimes_{E_*} E_*(X)
\]
defines a homology theory for all spaces $X$, since homology commutes with filtered colimits. However, we prefer to work cohomologically so that our theories are naturally ring-valued; here, things are a bit more complicated cohomologically. Base change provides a cohomology theory defined on finite spaces (spaces equivalent to finite CW-complexes) $R^* \otimes_{E^*} E^*(X)$ and on these spaces this is the same as $R^*(X)$. 

We may extend this to finite $G$-CW complexes Borel equivariantly: $R^* \otimes_{E^*} E^*(EG\times_G X)$, but there can be a large difference between $R^*(Y)$ and $R^* \otimes_{E^*} E^*(Y)$ for infinite $Y$.

\begin{example}
Let $E  = \E$, $R = p^{-1}\E$, and $Y = B\Z/p$. Then $\E^*(B\Z/p)$ is a free $\E^*$-module of rank $p^n$. Thus we see that $p^{-1}\E^* \otimes_{\E^*} \E^*(B\Z/p)$ is a free module of rank $p^n$ over $p^{-1}\E^*$. However, $(p^{-1}\E)^*(B\Z/p) \cong p^{-1}\E^*$ as $p^{-1}\E$ is a rational cohomology theory. 
\end{example}

The key observation is that, in general, $R^* \otimes_{E^*} E^*(-)$ does not satisfy the infinite wedge axiom. However, Brown representability (in the form of \cite{AdamsRepresentability}) applied to this theory defined on finite spaces produces a spectrum $R'$. 

\begin{lemma}\mylabel{landweberequal}
With the above notation, $R \simeq R'$.
\end{lemma}
\begin{proof}
The cohomology theory associated to $R'$ must take the same value as the cohomology theory associated to $R$ on finite spaces. Now this lemma is an immediate consequence of \cite[Thm. 2.8]{hoveystrickland}, which says that Landweber exact spectra are determined by their coefficients. 
\end{proof}

In the following, we will also need the fact that the smash product of even Landweber exact theories is again even. Note that this fails for general spectra, as the example $H\F_p \sm H\F_p$ shows. 

\begin{lemma}\mylabel{landwebereven}
Suppose that $E$ and $F$ are Landweber exact theories, then so is $E \sm F$. Additionally, if $E$ and $F$ are even, then $E \sm F$ is even as well. 
\end{lemma}
\begin{proof}
The first part of the claim follows from Hopkins discussion in \cite{coctalos}. Indeed, there is a pullback diagram
\[\xymatrix{\Spec(\pi_*E \sm F) \ar[r] \ar[d] & \Spec(F_*) \ar[d]^v \\
\Spec(E_*) \ar[r]_u & \mathcal{M_{\mathrm{FG}}}          }\]
where $\mathcal{M}_{fg}$ denotes the stack of one-dimensional formal groups. Since $u$ and $v$ are flat, the composite $\Spec(\pi_*E \sm F) \to \mathcal{M}_{fg}$ is flat by base-change. 
To show that $E \sm F$ is even, recall that
\[\pi_*(E \sm F) \cong E_* \otimes_{MU_*} MU_*MU \otimes_{MU_*} F_*\]
so the claim follows since $MU_*MU$ is concentrated in even degrees. 
\end{proof}

\section{Some spectra related to character theory}
\subsection{Recollections}
We recall the character maps of \cite{tgcm}. For the rest of the paper fix a prime $p$. Let $\E$ be Morava $E$-theory and $\LE = L_{K(t)}\E$ be the localization of $\E$ by Morava $K(t)$. By \myref{flatcoefficients} there is an isomorphism
\[
\pi_0\LE \cong \W k\llbracket u_1,\ldots,u_{n-1}   \rrbracket[u_t^{-1}]_{I_t}^{\wedge}.
\]
Let $\G_{\E}$ be the \pdiv group associated to $\E$ and $\G_{\LE}$ the \pdiv group associated to $\LE$. In \cite{tgcm} a flat extension $C_{t}^*$ of $\LE^*$ is constructed with the following property:

Let $\G := \LE^0 \otimes \G_{\E}$ so that $\G_{\LE}$ is the connected component of the identity of $\G$ (Proposition 2.4 of \cite{tgcm}). Recall the following proposition.
\begin{proposition}\mylabel{num2} (Proposition 2.17 in \cite{tgcm})
The functor from $\LE^0$-algebras to sets given by
\[
\Iso_{\G_{\LE} / }(\G_{\LE}\oplus\QZ^{n-t},\G) \colon R \mapsto \Iso_{\G_{\LE} / }(R \otimes \G_{\LE}\oplus\QZ^{n-t},R \otimes \G)
\]
is representable by $C_{t}^0$.
\end{proposition}

The ring $C_{t}^*$ is constructed as a localization of the ring $\Colim{k} \text{ } \LE^* \otimes_{\E^*} \E^*(B(\Z/p^k)^{n-t})$. The following result should be compared to Proposition 6.5 of \cite{hkr}.

\begin{proposition} \mylabel{p:ff}
The ring $C_{t}^{0}$ is a faithfully flat $\LE^0$-algebra.
\end{proposition}
\begin{proof}
Note that $C_{t}^0$ is a flat $\LE^0$-algebra since it is constructed as a colimit of algebras each of which is a localization of a finitely generated free module. 

We now show that it is surjective on $\Spec(-)$. Let $P \subset \LE^0$ be a prime ideal. Let $i$ be the smallest natural number such that $u_i \notin P$. Note that $i \leq t$. Now consider 
\[
K := \overline{(\LE^0/P)}_{(0)}
\]
the algebraic closure of the fraction field. There is an isomorphism (pg. 34, \cite{Dem})
\[
K \otimes \G \cong \G_{\text{for}} \oplus \QZ^{n-i}.
\]
The formal part has height $i$ because $u_i$ has been inverted. Now since
\[
\G_{\text{for}} \oplus \QZ^{n-i} \cong \G_{\text{for}} \oplus \QZ^{t-i} \oplus \QZ^{n-t},
\]
\myref{num2} implies that this is classified by a map $C_{t}^0 \lra{q} K$ that extends the canonical map $\LE^0 \lra{} K$. Now the kernel of $q$ must restrict to $P$.
\end{proof}

The ring $C_{t}^*$ is used in the construction of the transchromatic generalized character maps of \cite{tgcm}. For a finite $G$-CW complex $X$, let 
\[
\Fix_h(X) = \Coprod{\al \in \hom(\Z_{p}^{h},G)}X^{\im \al}.
\]
This is a finite $G$-CW complex with $G$-action given by $x \in X^{\im \al} \mapsto gx \in X^{\im g\al g^{-1}}$. The character map is a map $\E^*(EG \times_G X) \lra{} C_{t}^* \otimes_{\LE^*} \LE^*(EG \times_G \Fix_{n-t}X)$ with the following property:
\begin{theorem} \mylabel{charactermap} \cite{tgcm}
The character map has the property that the map induced by tensoring the domain up to $C_t$
\[
C_{t}^* \otimes_{\E^*}\E^*(EG \times_G X) \lra{\cong} C_{t}^* \otimes_{\LE^*}\LE^*(EG \times_G \Fix_{n-t}X)
\]
is an isomorphism.
\end{theorem}

\subsection{Some spectra related to character theory}
Let $C_t$ be the spectrum \[S^{-1}\Colim{k} \text{ } \LE \wedge_{\E} \E^{B\Lk}.\] It is clear that the coefficients of $C_t$ is the ring $C_{t}^*$ from the previous section. 

\begin{proposition}
The spectrum $C_t$ is $E(t)$-local.
\end{proposition}
\begin{proof}
Note that $\LE \wedge_{\E} \E^{B\Lk}$ is $E(t)$-local. In fact, it is $K(t)$-local as it is equivalent to $L_{K(t)}(\E^{B\Lk})$. This follows from the fact that $\E^{B\Lk}$ is a free $\E$-module spectrum. Now colimits in the $E(t)$-local category may be computed in the category of spectra (since localization with respect to $E(t)$ is smashing) so $C_t$ is $E(t)$-local.
\end{proof}

While $C_t \sm E_t$ is $E(t)$-local and flat as a $C_t$-module and as an $E_t$-module, it is not $K(t)$-local and thus the argument of \myref{compcbaregood} cannot be used. For that reason we introduce a variant $\bC$, which allows us to exploit the good finiteness properties of the $K(t)$-local category.

\begin{definition}
We define $\bC := L_{K(t)}(C_t \wedge E_t)$. 
\end{definition}

\begin{proposition}
The spectrum $\bC$ is even periodic and $\bC^{*}$ is faithfully flat as an $E_{t}^{*}$-module.
\end{proposition}
\begin{proof}
As seen in \myref{landwebereven}, the smash product of even periodic Landweber exact spectra is even periodic and this cannot be changed by completion (which is all that $K(t)$-localization is for $E(t)$-local $BP$-modules). By even periodicity it suffices to prove that $\bC^0$ is faithfully flat over $E_{t}^0$. Note that $\pi_0(C_t \wedge E_t)$ is flat as a $C_{t}^0$-module and as an $E_{t}^0$-module by Hopkins' result in \cite{coctalos}. By \myref{flatflat}, the completion of a flat $E_{t}^0$-module at $I_t$ is flat. Thus $\bC^0$ is flat as an $E_{t}^0$-module. For faithful flatness it suffices to prove that $\bC^0/I_t$ is non-zero. But Hovey-Strickland implies that it is $\pi_0(K(t) \wedge C_t)$ and since $C_{t}^0/I_t$ is nonzero (Proposition 2.17 in \cite{tgcm}) we know that the $K(t)$-localization of $C_t$ is nonzero.   
\end{proof}

\section{From $E$-theory to $E$-theory}
In this section we present a modification of the character maps of \cite{tgcm}. We begin by recalling the character maps. An upshot of the presentation here is that the character map is a map of $E_{\infty}$-rings. We then analyze the modification of the character map applied to good groups and use this to show that centralizers of tuples of commuting elements in good groups are good.
\subsection{Character maps written spectrally}
The character maps of \cite{tgcm} admit an obvious spectral interpretation. For a finite group $G$ and a finite $G$-CW complex $X$ we may consider the evaluation map $B\Lk \times EG \times_G \Fix_{n-t}(X) \lra{} EG \times_G X$. This induces the map of spectra
\[
\E^{EG \times_G X} \lra{} \E^{B\Lk} \wedge_{\E} \E^{EG \times_G \Fix_{n-t}(X)}.
\]
Now the canonical map $\E^{B\Lk} \lra{} C_t$ and the map $\E \lra{} \LE$ induces
\[
\E^{B\Lk} \wedge_{\E} \E^{EG \times_G \Fix_{n-t}(X)} \lra{} C_t \wedge_{\LE} \LE^{EG \times_G \Fix_{n-t}(X)}.
\]
After extending coefficients in the domain the composite induces an equivalence
\[
C_t \wedge_{\E}\E^{EG \times_G X} \lra{\simeq} C_t \wedge_{\LE} \LE^{EG \times_G \Fix_{n-t}(X)}.
\]
In all of this discussion we are merely using the flatness of $C_{t}^*$ over $\E^*$ and $\LE^*$ to translate the algebraic results of \cite{tgcm} to these spectral statements. It is worth noting that all of the maps above are $E_{\infty}$ (after choosing an $E_{\infty}$-inverse to the K\"unneth map). Thus it is clear that the character map is an equivalence of $E_{\infty}$-rings.

Now we present a modification of the above map that has some desirable properties. In particular, the codomain is related to $E_t$ in the same way that the above map is related to $\LE$. It seems to suffer in two respects though. It does not seem to induce an equivalence for all spaces after base change of the domain to $\bC$ and it is not as computable as the above map.

The modification is the composite of two maps. The first is the character map above. The second is the canonical map of $E_{\infty}$-rings
\[
C_t \wedge_{\LE} \LE^{EG \times_G \Fix_{n-t}(X)} \lra{} \bC^{EG \times_G \Fix_{n-t}(X)}.
\]

\begin{proposition}\mylabel{compcbaregood}
For any finite $G$-CW complex $X$ the canonical map $\bC \wedge_{E_t} E_{t}^{EG \times_G X} \lra{\simeq} \bC^{EG \times_G X}$ is an equivalence. 
\end{proposition}
\begin{proof}
Because $\bC^*$ is faithfully flat over $E_{t}^*$ there is an isomorphism
\[
\bC^* \otimes_{E_{t}^*} E_{t}^*(EG \times_G X) \cong \pi_{-*}(\bC \wedge_{E_t} E_{t}^{EG \times_G X}).
\]
It is clear that $\bC \wedge_{E_t} E_{t}^{EG \times_G X}$ is $E(t)$-local. However, since $E_{t}^*(EG \times_G X)$ is finitely generated and $E_{t}^*$ is Noetherian, the above isomorphism implies that
\[
\bC^* \otimes_{E_{t}^*} E_{t}^*(EG \times_G X)
\]
is $I_t$-complete and thus $\bC \wedge_{E_t} E_{t}^{EG \times_G X}$ is $K(t)$-local. Let $D(EG\times_G X) = F(EG\times_G X, L_{K(t)}S)$. Now we have equivalences
\begin{align*}
\bC \wedge_{E_t} E_{t}^{EG \times_G X} &\simeq L_{K(t)}(\bC \wedge_{E_t} E_{t}^{EG \times_G X}) \\
&\simeq  L_{K(t)}(\bC \wedge_{E_t} L_{K(t)}(E_{t}\wedge (D(EG \times_G X)))) \\
&\simeq L_{K(t)}(\bC \wedge_{E_t} E_{t}\wedge (D(EG \times_G X))) \\
&\simeq L_{K(t)}(\bC \wedge (D(EG \times_G X))) \\
&\simeq \bC^{EG \times_G X}.
\end{align*}
The second and fifth equivalences follow from the $K(t)$-local duality of spaces of the form $EG \times_G X$ (Corollary 8.7 in \cite{hoveystrickland}). 
\end{proof}

Thus in the category of $E_{\infty}$-rings we have the map
\[
\E^{EG \times_G X} \lra{} \bC \wedge_{E_t} E_{t}^{EG \times_G \Fix_{n-t}(X)}
\]
that factors
\[
\E^{EG \times_G X} \rightarrow C_t \wedge_{\LE} \LE^{EG \times_G \Fix_{n-t}(X)} \rightarrow \bC^{EG \times_G \Fix_{n-t}(X)} \overset{\simeq}{\leftarrow} \bC \wedge_{E_t} E_{t}^{EG \times_G \Fix_{n-t}(X)}.
\]
The middle map is the most mysterious. The reason for this is that it is not clear at all that $\bC^0$ is a flat $\LE^0$-module.

\begin{proposition}
Let $G$ be a finite group and let $X$ be a finite $G$-CW complex with the property $\LE^*(EG \times_G X)$ finitely generated and projective as an $\LE^*$-module, then there is an isomorphism
\[
\bC^* \otimes_{\LE^*} \LE^*(EG \times_G X) \cong \bC^*(EG \times_G X).
\]
\end{proposition} 
\begin{proof}
This proof is essentially the same as the proof of \myref{compcbaregood}. Since $\LE^*(EG \times_G X)$ is projective we have an isomorphism
\[
\bC^* \otimes_{\LE^*} \LE^*(EG \times_G X) \cong \pi_{-*}(\bC \wedge_{\LE} \LE^{EG \times_G X}).
\]
Since $\LE^*(EG \times_G X)$ is finitely generated and $\LE^*$ is Noetherian the smash product is $K(t)$-local. Now we have the same set of equivalences as in \myref{compcbaregood} with $E_t$ replaced by $\LE$.
\end{proof}

\begin{remark}
When $G = \Z/p^k$ we recover the anticipated result that $\bC^0 \otimes \G_{\LE}[p^k] \cong \bC^0 \otimes \G_{E_t}[p^k]$.
\end{remark}

\begin{definition}
A finite group $G$ is good (at the fixed prime $p$) if $\E^*(BG)$ is even and free for all $n$. 
\end{definition}

\begin{remark}
Our definition of a good group differs somewhat from the original in Definition 7.1 in \cite{hkr}. They observe that their definition implies the definition of a good group used here.
\end{remark}

\begin{remark}
Because the $\E$-cohomology of a good group is even it admits an algebro-geometric interpretation. Because the $\E$-cohomology of a good group is free the character map of \cite{hkr} is an embedding and so the ring can be attacked using character theoretic methods. 
\end{remark}

\begin{corollary}\mylabel{compcbarlgood}
Let $G$ be a good group and let $\al:\Z_{p}^{h} \lra{} G$, then 
\[
\bC^* \otimes_{\LE^*} \LE^*(BC(\im \al)) \cong \bC^*(BC(\im \al)).
\] 
\end{corollary}
\begin{proof}
We show that $\LE^*(BC(\im \al))$ is finitely generated and projective. The character map of \myref{charactermap} gives a factorization
\[
C_{t}^* \otimes_{\E^*} \E^*(BG) \cong \Prod{[\al] \in \hom(\Z_{p}^{n-t}, G)/\sim} C_{t}^* \otimes_{\LE^*} \LE^*(BC(\im \al)).
\]
For a fixed $\al$, this implies that $C_{t}^* \otimes_{\LE^*} \LE^*(BC(\im \al))$ is projective because it is a summand of a free module and finitely generated because there is a retract to the inclusion as a summand. Now since $C_{t}^*$ is faithfully flat as an $\LE^*$-module, faithfully flat descent for finitely generated projective modules implies that $\LE^*(BC(\im \al))$ is finitely generated and projective.
\end{proof}

Together these results give the theorem. 

\begin{theorem}
For a good group $G$ we have an equivalence
\[
\bC \wedge_{\E} \E^{BG} \lra{\simeq} \bC^{\Loops^{n-t}BG} \overset{\simeq}{\longleftarrow} \bC \wedge_{E_t} E_{t}^{\Loops^{n-t}BG},
\]
where $\Loops^{h}BG = \hom(B\Z_{p}^{h}, BG) = EG \times_G \Fix_{h}(\ast)$.
\end{theorem}

\begin{remark}
When $n=1$ we obtain a map from $E$-theory to $p$-adic $K$-theory. This seems like a useful tool. It allows one to reduce certain computations at height $n$ to computations in representation theory. This is used by the Tomer Schlank and the second author in \cite{tpst} to give a new proof and generalization of Strickland's theorem regarding the $E$-theory of symmetric groups.
\end{remark}

\section{Examples of good groups}

A comprehensive list of finite groups that are known to be good at a fixed prime $p$ can be found in the habilitation thesis \cite{SchusterHab} of Bj{\"o}rn Schuster, to which we refer for the original references; these include:
\begin{enumerate}
 \item abelian groups
 \item symmetric groups
 \item $GL_n(\F_q)$ with $p \nmid q$
 \item all groups of order $p^3$ and of order $32$
 \item metacyclic groups
 \item the Mathieu group $M_{12}$, see \cite{SchusterM12}.
\end{enumerate}
The collection $\Good$ of good groups is closed under products and also under wreath products with $\Z/p$; moreover, a group is good if its Sylow $p$-subgroup is good. If $G = H_1 \times H_2$, then $H_1$ is good if both $G$ and $H_2$ are good because there is a K{\"u}nneth isomorphism for good groups. Furthermore, given an extension of the form 
\[\xymatrix{\ast \ar[r] & H \ar[r] & G \ar[r] & \Z/p \ar[r] & \ast,}\]
Kriz \cite{Krizodd} gives conditions for when $H$ good implies $G$ is good, and conversely. In particular, semi-direct products of elementary abelian $p$-groups and $\Z/p$ are good.

The methods of the previous section allow us to deduce a new closure property of $\Good$.

\begin{corollary}\mylabel{goodcentralizers}
Let $G$ be a good group and let $\al:\Z_{p}^{h} \lra{} G$, then $\E^*(BC(\im \al))$ is finitely generated, free, and evenly generated as an $\E^*$-module.
\end{corollary}
\begin{proof}
\myref{compcbarlgood} implies that $\bC^*(BC(\im \al))$ is finitely generated, projective, and even. \myref{compcbaregood} implies that 
\[
\bC^*(BC(\im \al)) \cong \bC^* \otimes_{E_{t}^*} E_{t}^*(BC(\im \al))
\]
and \myref{p:ff} implies that $C_{t}^*$ is faithfully flat as an $E_{t}^*$-modules. Faithfully flat descent for finitely generated projective modules implies that $E_{t}^*(BC(\im \al))$ is finitely generated, even, and projective. But now since $E_{t}^*$ is complete local, projective implies free.
\end{proof}

\begin{remark}
Calling a group $E_n$-good if its $E_n$-cohomology is concentrated in even degrees and free, \myref{goodcentralizers} applied to the identity element shows that $E_{n+1}$-good implies $E_n$-good. This is compatible with Minami's result in \cite{Minamikn}.
\end{remark}

Using GAP, we can thus construct new examples of good groups. 
\begin{example}
Let $p=2$ and consider $G = GL_2(\F_3) \wr C_2$, a good group of order $4608$. There exists an element $g \in G$ of order $4$ with centralizer 
\[C_G(g) = H \rtimes C_2, \] 
where $H$ the binary octahedral group, i.e., a non-split extension of $S_4$ by $C_2$. The GAP-Id of $C_G(g)$ in the Small Groups Library is $[96,192]$. Since the Sylow 2-subgroup of $C_G(g)$ has order 32, \cite{Schuster32} independently shows that this group has to be good. 

However, the group $K = C_G(g) \wr C_2$ contains an element $k$ of order $8$ whose centralizer 
\[C_K(k) = (((C_8 \times C_2) \rtimes C_2) \rtimes C_3) \rtimes C_2\]
has GAP-Id $[192,963]$ and Sylow 2-subgroup $(C_8 \times C_4) \rtimes C_2$, which is therefore not covered by the previous list of examples. This process can be iterated, giving rise to other new examples. 
\end{example}

For an odd prime $p$, \cite{Kriz-odd} implies that the unipotent radical in $GL_4(\F_p)$ is not good. So, as a curious consequence, we see that it cannot be obtained by iteratively applying the constructions $\Good$ is closed under to the above list of known good groups. We do not know how to show this using only algebraic methods.

\bibliographystyle{alpha}
\bibliography{mybib}

\end{document}